\documentclass{article}
\usepackage[T1]{fontenc}
\usepackage{amssymb}
\usepackage{amsmath}
\usepackage{fancyhdr}
\usepackage{graphicx}
\usepackage{dsfont}
\usepackage{amsthm}
\usepackage{pictexwd,dcpic}
\usepackage{enumerate}
\usepackage{mathtools}

\newtheorem{lemma}{Lemma}
\newtheorem{proposition}{Proposition}
\newtheorem{theorem}{Theorem}
\newtheorem{corollary}{Corollary}

\newcommand{\Spec}{\text{Spec  }}
\newcommand{\Spf}{\text{Spf  }}

\theoremstyle{definition}

\newtheorem{remark}{Remark}

\begin{document}

\title{A Monodromy criterion for the good reduction of $K3$ surfaces}
\author{Genaro Hern\'andez Mada}

\maketitle
\begin{abstract}
\noindent
We give a criterion for the good reduction of semistable $K3$ surfaces over $p$-adic fields using purely $p$-adic methods. We use neither $p$-adic Hodge theory nor transcendental methods as in the analogous proofs of criteria for good reduction of curves or $K3$ surfaces. We achieve our goal by realizing the special fiber $X_s$ of a semistable model $X$ of a $K3$ surface over the $p$-adic field $K$, $X_K$ as a special fiber of a log-family in characteristic $p$ and use an arithmetic version of the Clemens-Schimd exact sequence in order to obtain a Kulikov-Persson-Pinkham classification theorem in characteristic $p$.
\end{abstract}



{\bf Acknowledgement: } This work is a part of the author's PhD thesis. It has been supported by an Erasmus Mundus ALGANT-DOC scholarship and was done at the Universities of Padova, Concordia and Bordeaux. The author has been partially supported by the MIUR-PRIN 2010-11 grant Arithmetic Algebraic Geometry and Number theory. \newline I would especially like to thank my main supervisors Bruno Chiarellotto and Adrian Iovita.

{\bf Comments: } Submitted to Mathematische Zeitschrift.

\section{Introduction}
Let $p>0$ be a prime integer and $K$ a finite extension of $\mathbb Q_p$. Consider a smooth, proper and geometrically irreducible scheme $X_K$ over $\Spec K$. The question of whether $X_K$ has good reduction or not can be answered via $\ell$-adic or $p$-adic criteria in some cases. For example, if $X_K=A_K$ is an abelian variety and $G_K$ the absolute Galois group of $K$ we have that $A_K$ has good reduction if and only if for all $\ell \neq p$ (equivalently, for some $\ell \neq p$), the $\ell$-adic $G_K$-representation $T_\ell (A_K)$ is unramified (see \cite[theorem 1]{[ST68]}). The $p$-adic criterion says that $A_K$ has good reduction if and only if the $p$-adic $G_K$-representation $T_p(A_K)$ is crystalline (see \cite[Theorem II.4.7]{[CI99]} and \cite[Corollaire 1.6]{[Br00]}). Recall that in general for semistable $p$-adic representations this is equivalent to having trivial monodromy operator.\newline \indent
For more general varieties the criteria from the preceding paragraph are not valid, but in some cases, different criteria can be obtained. For example if $X_K$ is a curve with semistable reduction, Oda (in \cite[Section 3]{[Oda95]}) obtained an $\ell$-adic criterion looking at the Galois action on the \'etale fundamental group (via an analogous trascendental result) and Andreatta-Iovita-Kim (\cite[Theorem 1.6]{[AIK13]}) obtained its $p$-adic version studying the monodromy action on the De Rham unipotent fundamental group (via $p$-adic Hodge Theory). This means that it is not enough to look at the first cohomology group with its Galois/monodromy action but one needs to look at the whole (unipotent) fundamental group (i.e., not only its abelianization). \newline \indent
In this article we obtain a $p$-adic criterion for $K3$ surfaces. Namely we suppose that $p>3$ and $X_K$ is a smooth, projective $K3$ surface over $\Spec K$ having a minimal semistable model $X$ over the ring of integers $O_K$ of $K$. We may assume to have combinatorial reduction (see \cite[Proposition 3.5]{[Na00]}). Then since we are dealing with $K3$ surfaces, the first De Rham cohomology group is trivial, as well as the connected De Rham fundamental group. Therefore we look at the monodromy action on the second De Rham cohomology group $H^2_{\textrm{DR}}(X_K)$.  This operator $N$ (the mondromy that is) is described in the frameweork of the theory of log-schemes and log-crystalline cohomology (see \cite[Theorem 5.1]{[HK94]}).  Then our result  is  the following: 

\begin{center}
{\it Under the hypotheses above, the $K3$ surface $X_K$ has good reduction if and only if the monodromy $N$ is zero on $H^2_{\textrm{DR}}(X_K)$.}
\end{center}
 
In fact we obtain more than that. We know that the operator $N$ is always nilpotent ($N^3$ is always trivial). In the case the reduction of $X$ is not good we can refine the theorem just stated: the type of bad reduction is determined by the order of nilpotency of $N$. For the complete result, see theorem \ref{mainthm} .\newline \indent
One can also note that this criterion for good reduction in terms of the monodromy operator on log-crystalline cohomology can also be formulated in  ``\'etale terms''. Namely,  $X_K$ has good reduction if and only if $H_{\textrm{\'et}}^2(X_{\overline{K}}, \mathbb Q_p)$ is a crystalline representation. This is a consequence of our criterion and the comparison theorem \cite[Theorem 0.2]{[Ts99]}.\newline \indent
In the classical situation (over the complex numbers), given a semistable degeneration of $K3$ surfaces the work of Kulikov \cite{[Ku77]}, Persson-Pinkham \cite{[PP81]} and Morrison \cite{[Mo84]} show how the monodromy action on the generic fiber determines the behavior of the special one. To prove this one uses all the information coming from the structure of the family: the weight-monodromy conjecture and the Clemens-Schmid exact sequence. Our proof has been inspired by these methods. \newline \indent
The monodromy on the De Rham cohomology of $X_K$  is given by the monodromy operator on the log-crystalline cohomology of the special fiber $X_s$ (which is a characteristic $p$-scheme) endowed with the induced log-structure (see for example \cite[Theorem 5.1]{[HK94]}). Using Nakkajima's results on deformations of $K3$ surfaces (\cite{[Na00]}) we may construct a log-smooth deformation of our special fiber over the ring of formal power series $k[[t]]$, where $k$ is the residue field of $K$. Then using  Popescu's version of Artin's approximation theorem we get a deformation of $X_s$ over a smooth scheme $Y$ over $k[t]$ (possibly of dimension larger than 1). Finally by taking a well-chosen curve inside $Y$ we are reduced to the case of a family over a smooth curve, so we can use Chiarellotto-Tsuzuki's results for this setting. In particular for such a family we can use the weight-monodromy conjecture and the existence of a Clemens-Schmid type exact sequence. This gives the elements to rephrase Kulikov-Persson-Pinkham's and Morrison's results in characteristic $p$, allowing us to get our main theorem (theorem \ref{mainthm}) which is similar to the one obtained by P\'erez Buend\'ia in \cite{[Pe14]}. \newline \indent
This method of proof is completely different to the one used by Matsumoto in \cite{[Ma14]}, who also obtains results similar to ours, but for a different case, allowing algebraic spaces as models and always working with the generic fiber.\newline \indent
One might have hoped to use our methods to study the case of semistable Enriques surfaces. More precisely, we work with a semistable model, under some hypothesis about the canonical bundle over the DVR. This is compatible with \cite{[LM14]}, where again they study only the generic fiber. Then, we can follow our techniques along the lines we used for $K3$ surfaces since again by Nakkajima's work \cite{[Na00]} we have a classification of the possible special fibers (see the beginning of our section \ref{AppK3} to see the complete classification). But when we apply Morrison's methods we find that the monodromy is zero on the second cohomology group without any connection with the fact that the special fiber is smooth or not. We plan to investigate this problem in another article (see also the parallel work \cite{[LM14]}). \newline \indent

Let us give an outline of this article. In sections \ref{Notation}-\ref{Reduction} we get a generalization of Clemens-Schmid exact sequence in characteristic $p$ and in the last two sections we use it to prove our main result. More precisely, in section \ref{Notation} we establish our notation and setting.

In section \ref{Popescu}, we use N\'eron-Popescu desingularization (see \cite[Theorem 1.1]{[Sw95]}) to write the ring of formal power series $k[[t]]$ as a limit of smooth $k[t]$-algebras:
\[
k[[t]] = \lim_{\stackrel{\longrightarrow}{\alpha}} A_\alpha,
\]
and this allows, in a similar way to what is done in \cite[Section 4]{[It05]}, to see our situation as a fiber inside a larger family of varieties $f:X_A\rightarrow Y=\Spec A$, where $A=A_\alpha$ for some $\alpha$. Then we can use the relative cohomology theories defined and studied by Shiho \cite{[Sh08]}, which give relative cohomology sheaves on a formal scheme $\mathcal Y$, that is a smooth lifting of $Y$. This will be useful once we have a deformation, which is done in section \ref{AppK3}. \newline \indent
In section \ref{Cohomology} we state the results on relative cohomology that are useful for our purposes. In particular we need the base-change theorem and the comparison isomorphisms between the different cohomology theories (log-crystalline, log-convergent and log-analytic), since these are needed to use the results in \cite{[CT12]}. This means that the relative cohomology sheaves, defined on the large family, satisfy the desired properties. \newline \indent
Then, in section \ref{Reduction} we construct a smooth curve $C$ inside $Y$ in such a way that we can restrict the family $X_A \rightarrow Y$, as well as the cohomology sheaves, to a smaller family $X_C$ over this curve. In particular this allows us to use the main result in \cite{[CT12]} and get the version of the Clemens-Schmid exact sequence for this setting:

\[
\cdots \rightarrow H_{rig}^m(X_s) \rightarrow H_{log-crys}^m((X_s,M_s)/W^\times)\otimes K_0 \stackrel{N}{\rightarrow}  
\]
\[
H_{log-crys}^m((X_s,M_s)/W^\times)\otimes K_0(-1)\rightarrow H_{X_s,rig}^{m+2}(X_C) \rightarrow H_{rig}^{m+2}(X_s)\rightarrow \cdots
\]

In section \ref{MonCrit} we use the Clemens-Schmid exact sequence in characteristic $p$ to get criteria for $N$ to be the zero map on $H_{log-crys}^1$ or $H_{log-crys}^2$, assuming that we are dealing with a semistable family of varieties over a smooth curve over a finite field. For this we use the fact that the monodromy and weight filtrations on the special fiber coincide (as proved in \cite{[CT12]}, using the fact that we deal with a family of varieties). For a more general situation, i.e., if we do not assume that the special fiber is inside a semistable family of varieties, this is known only for the case of curves and surfaces (see \cite[Sections 5 and 6]{[Mk93]}). The criteria that we get in this section are in terms of the Betti numbers of the dual graph of the special fiber which can be easily described in the case of combinatorial reduction. As we mentioned before we can always restrict ourselves to this case after a finite base extension. \newline \indent
Finally in section \ref{AppK3}, after introducing the deformation theory for $K3$ surfaces along the lines of Nakkajima \cite{[Na00]}, we apply the criteria from section \ref{MonCrit} to the case of $K3$ surfaces, assuming that the special fiber is combinatorial, i.e., it is one of three possible types. We obtain that the degree of nilpotency determines the type of degeneracy of the special fiber. Our main result will be stated in theorem \ref{mainthm}: the trivial monodromy action on the second de Rham cohomological group is equivalent to good reduction.

\section{Notation and setting}\label{Notation}
In this article we fix a prime integer $p>0$ and $K$ a finite extension of $\mathbb Q_p$ with ring of integers $O_K$. We denote by $X_K$ a smooth, projective $K3$ surface over $K$. For a finite field $k$, we denote by $W=W(k)$ its ring of Witt vectors and $K_0$ the fraction field of $W$. We shall denote by the same letter $W$ the formal scheme Spf $W$ with the trivial log structure, and we denote by $W^\times$ the same formal scheme with the log structure given by $1\mapsto 0$.\newline \indent
Recall that a divisor $Z\subset Y$ of a noetherian scheme is said to be a strict normal crossing divisor (SNCD) if $Z$ is a reduced scheme and, if $Z_i$, $i\in J$ are the irreducible components of $Z$, then, for any $I\subset J$, the intersection $Z_I=\cap_{i\in I}Z_i$ is a regular scheme of codimension equal to the number of elements of $I$. We shall say that $Y$ is a normal crossing divisor (NCD) if, \'etale locally on $Y$, it is a SNCD.\newline \indent
The first five sections are dedicated to obtaining a Clemens-Schmid type exact sequence. For this, we consider a proper and flat morphism $F:X\rightarrow \Spec k[[t]]$ over $k$, where $X$ is a smooth scheme such that \'etale locally it is \'etale over $$\textrm{Spec}\left( k[[t]][x_1,...,x_n]/(x_1\cdots x_r -t)\right).$$
We denote by $s$ the closed point of $\Spec k[[t]]$ and $X_0$ its fiber, which is a NCD inside $X$. We denote by $(X,M)$ the scheme $X$ endowed with the log structure defined by $X_0$, $\Spec k[[t]]^\times$ the scheme $\Spec k[[t]]$ endowed with the log structure defined by the point $s$ (i.e., by the NCD given by the ideal generated by $t$), and $s^\times$ the log point given by the point $s$ and the log structure induced from $\Spec k[[t]]^\times$. 
Then, we have the following cartesian diagram of log schemes
\begin{equation}\label{initial}
\begindc{0}[700]
\obj(0,1)[X0]{$(X_0,M_0)$}
\obj(0,0)[kk]{$s^\times$}
\obj(1,1)[XX]{$(X,M)$}
\obj(1,0)[kt]{$\text{Spec } k[[t]]^\times$}
\mor{XX}{kt}{$F$}[\atleft,\solidarrow]
\mor{X0}{XX}{}[\atright,\solidarrow]
\mor{X0}{kk}{}[\atright,\solidarrow]
\mor{kk}{kt}{}[\atleft,\solidarrow]
\enddc
\end{equation}
where $(X_0,M_0)$ is obtained by taking the fiber product in the category of log schemes.

\section{A construction using N\'eron-Popescu desingularization}\label{Popescu}
In order to get the desired result, we need to study the cohomology of the special fiber $X_0$ of $X$ over $k[[t]]$ and for this we first use the following theorem of Popescu (see  \cite[Theorem 1.1]{[Sw95]}):

\begin{theorem}
Let $f:R\rightarrow \Lambda$ be a morphism of rings. Then, $f$ is geometrically regular if and only if $\Lambda$ is a filtered colimit of smooth $R$-algebras.
\end{theorem}

It can be checked that the natural morphism $k[t] \rightarrow k[[t]]$ is geometrically regular according to the definition in \cite{[Sw95]}:
\begin{theorem}
The natural morphism $k[t]\rightarrow k[[t]]$ is geometrically regular.
\end{theorem}
\begin{proof}
It is clearly flat since it is a completion. Now there are only two prime ideals of $k[[t]]$. Namely, $0$ and $(t)$, and their respective counterpart in $k[t]$ are the only couples to consider in the definition.\newline \indent
Case 1 (the ideals generated by $t$): in this case, we need to check that $k\otimes_{k[t]}k[[t]]_{(t)}\cong k$ is geometrically regular over $k$, which is trivial.\newline \indent
Case 2 (the ideals 0): in this case, we need to check that $k(t) \otimes_{k[t]} k((t))$ is geometrically regular over $k(t)$. Take a finite extension $k'$ of $k(t)$ such that $(k')^p \subset k(t)$. Note that this is necessarily $k(t^{1/p})$. Indeed, it is a finite extension of degree $p$ (hence it does not have any subextension) and $(k')^p=k(t)$. Then, we only need to check that $k(t^{1/p}) \otimes _{k(t)}(k(t)\otimes_{k[t]} k((t)))$ is a regular local ring, but $$k(t^{1/p}) \otimes _{k(t)}(k(t)\otimes_{k[t]} k((t))) \cong k((t^{1/p})),$$
which is clearly regular.
\end{proof}

In particular,
\begin{equation}\label{ArtinApp}
k[[t]]=\lim_{\stackrel{\longrightarrow}{\alpha}} A_\alpha,
\end{equation}
where the $A_\alpha$'s are smooth $k[t]$-algebras. One can say even more:

\begin{proposition}\label{Iovitaremark}
Let $\overline k$ be an algebraic closure of $k$ and $A$ a smooth $\overline k[t]$-algebra. Then, there exists a finite extension $k'$ of $k$ and a smooth $k'[t]$-algebra $A'$ such that $A' \otimes_{k'} \overline k \cong A$. 
\end{proposition}
\begin{proof}
Let us take a presentation of $A$ of the type $\overline k[t][x_1,...,x_n]/(f_1,...,f_c).$ Since $f_1,...,f_c$ are a finite number of polynomials, one needs a finite number of elements of $\overline k$ to define them in the variables $t,x_1,...,x_n$. Let $k'$ be a finite extension of $k$ containing all of those coefficients and define $A' := k'[t][x_1,...,x_n]/(f_1,...,f_c)$. We only need to assure that $A'$ is smooth over $k'[t]$. This is a direct consequence of corollary 17.7.3, part ii), in EGAIV.
\end{proof}

Since $X$ is proper over $k[[t]]$, by (\ref{ArtinApp}) there exist a smooth $k[t]$-algebra, a scheme $X_A$, proper over $\text{Spec } A$, Zariski locally \'etale over $$\Spec A[x_1,...,x_n]/(x_1\cdots x_r-t), $$and such that the following diagram is cartesian:
\begin{equation}\label{ArtinAppr}
\begindc{0}[700]
\obj(0,1)[XX]{$X$}
\obj(0,0)[VV]{$\text{Spec } k[[t]]$}
\obj(1,1)[XA]{$X_A$}
\obj(1,0)[YY]{$\text{Spec } A$}
\mor{XX}{XA}{$u$}[\atleft,\solidarrow]
\mor{XX}{VV}{$F$}[\atright,\solidarrow]
\mor{VV}{YY}{$v$}[\atright,\solidarrow]
\mor{XA}{YY}{$f$}[\atleft,\solidarrow]
\enddc
\end{equation}
Note that the composition $v\circ F$ is flat, hence $f$ is flat on an open of $X_A$ containing the image of $X$ under $u$. Thus, we may assume that $f:X_A \rightarrow \text{Spec } A$ is flat.\newline \indent
Since the divisor of $Y=\text{Spec} A$, defined by $Y_0=(t=0)$ is a NCD, and the fiber product $X_{A,t=0}=Y_0\times_{Y} X_A$ is a NCD divisor in $X_A$, then we can naturally define fine log structures $M_A$ and $N$ on $X_A$ and $Y$, respectively. Then, $f:(X_A,M_A)\rightarrow (Y,N)$ is a morphism of log schemes. Moreover, we have the following:

\begin{lemma}
The morphism $f:(X_A,M_A)\rightarrow (Y,N)$ is log-smooth.
\end{lemma}
\begin{proof}
We use \cite[Theorem 3.5]{[Ka89]}. First note that $f$ has (\'etale locally on $X_A$) a chart $(P_{X_A}\rightarrow M_A, Q_Y\rightarrow N, Q\rightarrow P)$ given by $Q=\mathbb N$, $P=\mathbb N^r$, and the diagonal map $Q\rightarrow P$.\newline \indent
We can easily see also that the kernel and the torsion part of the cokernel of $Q^{gp} \rightarrow P^{gp}$ (which is just the diagonal map $\mathbb Z \rightarrow \mathbb Z^r$) are both trivial.\newline \indent
It remains to prove that the induced morphism $X_A \rightarrow Y \times_{\Spec \mathbb Z[Q]} \Spec \mathbb Z[P]$ is smooth. Recall that $X_A$ is locally \'etale over $$V=\Spec (A[x_1,...,x_n]/(x_1 \cdots x_r-t)),$$ and note that
\[
\Spec A \times_{\Spec \mathbb Z[Q]} \Spec \mathbb Z[P] \cong \Spec A \times_{\Spec \mathbb Z[u]} \Spec \mathbb Z[u_1,...,u_r]
\]
\[
\cong \Spec(A[u_1,...,u_r]/(u_1\cdots u_r -t)) =:W.
\]
The last isomorphism can be verified by checking directly that the ring $$A[u_1,...,u_r]/(u_1\cdots u_r -t)$$ satisfies the universal property of the tensor product $A\otimes_{\mathbb Z[u]} \mathbb Z[u_1,...,u_r]$.\newline \indent
Now note that there are natural closed immersions $j_V:V \hookrightarrow \mathbb A_A^n$, and $j_W:W \hookrightarrow \mathbb A_A^r$. Moreover, the following diagram is cartesian:

\[\begindc{0}[700]
\obj(0,1)[VV]{$V$}
\obj(0,0)[WW]{$W$}
\obj(1,1)[An]{$\mathbb A_A^n$}
\obj(1,0)[Ar]{$\mathbb A_A^r$}
\mor{VV}{An}{$j_V$}[\atleft,\solidarrow]
\mor{VV}{WW}{$h$}[\atright,\solidarrow]
\mor{WW}{Ar}{$j_W$}[\atright,\solidarrow]
\mor{An}{Ar}{$p$}[\atleft,\solidarrow]
\enddc\]
where $h$ is defined by sending each $u_i$ to $x_i$ for $i=1,...,r$, and $p$ is the natural projection from the first $r$ components. Since $p$ is smooth, we get that $h$ is smooth. Since $X_A\rightarrow W$ is the composition of an \'etale and a smooth morphism, we conclude that it is smooth (in the classical sense).
\end{proof}

Then we have the following diagram of log schemes:

\[\begindc{0}[700]
\obj(0,1)[SF]{$(X_0,M_0)$}
\obj(0,0)[Pt]{$s^\times$}
\obj(1,1)[XX]{$(X,M)$}
\obj(1,0)[DV]{$\Spec k[[t]]^\times$}
\obj(2,1)[XA]{$(X_A,M_A)$}
\obj(2,0)[YY]{$(Y,N)$}
\mor{SF}{Pt}{$f_s$}[\atright,\solidarrow]
\mor{XA}{YY}{$f$}[\atleft,\solidarrow]
\mor{Pt}{DV}{}[\atright,\solidarrow]
\mor{SF}{XX}{}[\atright,\solidarrow]
\mor{XX}{XA}{}[\atright,\solidarrow]
\mor{DV}{YY}{}[\atright,\solidarrow]
\mor{XX}{DV}{$F$}[\atright,\solidarrow]
\enddc\]
In particular, note that $s$ is a closed point inside $Y$, hence $(X_0,M_0)$ is a fiber of the log smooth family $(X_A,M_A)\rightarrow (Y,N)$. This means that we can study the cohomology of $X_0$ using relative cohomology sheaves for this family. These are studied in the next section.

\section{Relative Cohomology}\label{Cohomology}

By \cite[Theorem 7, Secion 4]{[El73]}, there exists a $W[t]$-algebra $A_0$ which is smooth over $W$ and such that $A_0/pA_0=A$.  Let $\hat A$ be the $p$-adic completion of $A_0$, and $\mathcal Y=\Spf \hat A$. We can define a log structure $\mathcal N$ on $\mathcal Y$ by $1\mapsto t$, and then we have the following diagram:
\begin{equation}\label{setting}
\begindc{0}[700]
\obj(0,1)[SF]{$(X_0,M_0)$}
\obj(0,0)[Pt]{$s^\times$}
\obj(1,1)[XA]{$(X_A,M_A)$}
\obj(1,0)[YY]{$(Y,N)$}
\obj(2,0)[FY]{$(\mathcal Y, \mathcal N)$}
\mor{SF}{XA}{}[\atleft,\solidarrow]
\mor{SF}{Pt}{$f_s$}[\atright,\solidarrow]
\mor{Pt}{YY}{}[\atright,\solidarrow]
\mor{XA}{YY}{$f$}[\atleft,\solidarrow]
\mor{YY}{FY}{}[\atleft,\solidarrow]
\enddc
\end{equation}
where the lower row consists of two exact closed immersions. Now we are in the situation studied in \cite{[Sh08]} and we can use all the results there. We shall state the results on relative log crystalline, log convergent and log analytic cohomology that are useful to apply the main result in \cite{[CT12]}.

\subsection{Relative Log Crystalline Cohomology}
In the situation of diagram (\ref{setting}), Shiho defined in \cite{[Sh08]}, for any sheaf $\mathcal F$ on the log crystalline site $(X/\mathcal Y)_{crys}^{log}$ the sheaves of relative log crystalline cohomology of $(X_A,M_A)/(\mathcal Y,\mathcal N)$ with coefficient $\mathcal F$, denoted by $R^m f_{X_A/\mathcal Y,crys*}\mathcal F$, and for an isocrystal $\mathcal E=\mathbb Q\otimes \mathcal F$, denoted by $R^m f_{X_A/\mathcal Y,crys*}\mathcal E$. Here we will work only with the trivial log isocrystal $\mathcal E=\mathcal O_{X_A/\mathcal Y, crys}$.\newline \indent
In order to study the sheaves $R^m f_{X_A/\mathcal Y, crys*}\mathcal O_{X/\mathcal Y,crys}$, we fix a Hyodo-Kato embedding system $(\mathcal P_{\bullet},\mathcal M_{\bullet})$ of an \'etale hypercovering $(X_{\bullet},M_{\bullet})$ of the log-scheme $(X_A,M_A)$. It always exists, as stated in \cite{[HK94]} (the definition of simplicial schemes and \'etale hypercoverings can be found in \cite[Section 7.2]{[CT03]}). Then we have the following diagram:

\begin{equation}\label{bigdiagram}
\begindc{0}[700]
\obj(0,0)[Pt]{$s^\times$}
\obj(0,1)[SF]{$(X_0,M_0)$}
\obj(0,2)[SFC]{$(X_{0,\bullet},M_{0,\bullet})$}
\obj(1,0)[YY]{$(Y,N)$}
\obj(1,1)[XA]{$(X_A,M_A)$}
\obj(1,2)[CO]{$(X_{\bullet},M_{\bullet})$}
\obj(2,0)[FY]{$(\mathcal Y,\mathcal N)$}
\obj(2,2)[FC]{$(\mathcal P_{\bullet},\mathcal M_{\bullet})$}
\mor{SFC}{CO}{}[\atleft,\solidarrow]
\mor{SFC}{SF}{$\theta_s$}[\atright,\solidarrow]
\mor{SF}{XA}{}[\atleft,\solidarrow]
\mor{SF}{Pt}{$f_s$}[\atright,\solidarrow]
\mor{Pt}{YY}{}[\atright,\solidarrow]
\mor{CO}{FC}{$i_{\bullet}$}[\atleft,\solidarrow]
\mor{CO}{XA}{$\theta$}[\atright,\solidarrow]
\mor{XA}{YY}{$f$}[\atleft,\solidarrow]
\mor{YY}{FY}{}[\atright,\solidarrow]
\mor{FC}{FY}{$g$}[\atleft,\solidarrow]
\enddc
\end{equation}
where $(X_{0,\bullet},M_{0,\bullet})$ is the fiber product in the upper left square.\newline \indent
We want to see that the sheaves $R^m f_{X_A/\mathcal Y,crys *}(\mathcal O_{X/\mathcal Y,crys})$ satisfy some finiteness properties. For each $n\in \mathbb N$, denote by $\mathcal Y_n$ the reduction of $\mathcal Y$ modulo $p^n$, and $C_{X_\bullet/\mathcal Y_n}$ the logarithmic De Rham complex of the log PD-envelope of the closed immersion $i_\bullet$ over $(\mathcal Y_n,\mathcal N_n)$. Then we have the following:

\begin{lemma}
\begin{enumerate}[(a)]
\item
For each $n$, there is a canonical quasi-isomorphism $$R(f\theta)_* C_{X_\bullet/\mathcal Y_n}\otimes_{\mathcal O_{\mathcal Y_n}}^L\mathcal O_{\mathcal Y_{n-1}}\stackrel{\sim}{\longrightarrow} R (f\theta)_*C_{X_\bullet/\mathcal Y_{n-1}}.$$
\item
For each $n$, $R(f\theta)_*C_{X_\bullet/\mathcal Y_n}$ is bounded and has finitely generated cohomologies.
\end{enumerate}
\end{lemma}
\begin{proof}
In \cite[Section 1]{[Sh08]}, it is proved that $$R(f\theta)_* C_{X_\bullet/\mathcal Y_n}\cong Rf_{X_\bullet/\mathcal Y_n,crys,*}(\mathcal O_{X_\bullet/\mathcal Y_n,crys}),$$ and so part (a) follows from the claim in the proof of theorem 1.15 in \cite{[Sh08]}.\newline \indent
For part (b), we proceed inductively. Note that for $n=1$, $\mathcal Y_1=Y$, and so the result follows by properness of $f$. The inductive step is direct using the second part of the same claim used in (a).
\end{proof}

The preceding lemma says that $\{R(f\theta)_*C_{X_\bullet/\mathcal Y_n}\}_n$ is a consistent system, as defined in \cite{[BO78]}. Then by \cite[Corollary B.9]{[BO78]}, it follows that $$Rf_{X_A/\mathcal Y,crys *}(\mathcal O_{X_A/\mathcal Y,crys})=R \lim_{\longleftarrow}Rf_{X_A/\mathcal Y_n,crys,*}(\mathcal O_{X_A/\mathcal Y_n,crys})$$ is bounded above and has finitely generated cohomologies. Thus we have the following:
\begin{theorem}\label{finitecrys}
The complex $Rf_{X_A/\mathcal Y,crys *} (\mathcal O_{X_A/\mathcal Y,crys})$ is a perfect complex of isocoherent sheaves on $\mathcal Y$. Moreover, the isocoherent cohomology sheaf $$R^mf_{X_A/\mathcal Y,crys *}(\mathcal O_{X_A/\mathcal Y,crys})$$ admits a Frobenius structure for each $m$.\end{theorem}
\begin{proof}
The first assertion follows from the above paragraph and \cite[Theorem 1.16]{[Sh08]}. The Frobenius structure is given by \cite{[HK94]}, since $f$ is of Cartier type. Indeed, recall that $f$ has a local chart $(P_{X_A}\rightarrow M_A,Q_Y\rightarrow N,Q\rightarrow P)$ given by $Q=\mathbb N$, $P=\mathbb N^r$, and $Q\rightarrow P$ the diagonal map.
\end{proof}

Now let us consider the following commutative diagram, where all squares are cartesian:

\begin{equation}\label{basechange}
\begindc{0}[700]
\obj(0,0)[XA]{$(X_A,M_A)$}
\obj(1,0)[YY]{$(Y,N)$}
\obj(0,1)[X0]{$(X_0,M_0)$}
\obj(1,1)[Pt]{$s^\times$}
\obj(2,0)[FY]{$(\mathcal Y,\mathcal N)$}
\obj(2,1)[FP]{$\Spf W^\times$}
\mor{X0}{XA}{}[\atleft,\solidarrow]
\mor{X0}{Pt}{$f_s$}[\atleft,\solidarrow]
\mor{XA}{YY}{$f$}[\atright,\solidarrow]
\mor{Pt}{YY}{}[\atright,\solidarrow]
\mor{Pt}{FP}{}[\atleft,\solidarrow]
\mor{YY}{FY}{$\iota$}[\atright,\solidarrow]
\mor{FP}{FY}{$\varphi$}[\atleft,\solidarrow]
\enddc
\end{equation}

By \cite[Theorem 1.19]{[Sh08]}, we have the following base change property.
\begin{theorem}
In diagram (\ref{basechange}), there is a quasi-isomorphism
\[
L\varphi^* Rf_{X_A/\mathcal Y,crys *}(\mathcal O_{X_A/\mathcal Y,crys}) \stackrel{\sim}{\longrightarrow} Rf_{s,X_0/W,crys,*}(\mathcal O_{X_0/W,crys}).
\]
\end{theorem}
Note that $Rf_{s,X_0/W,crys,*}(\mathcal O_{X_0/W,crys})$ is a perfect $K_0$-complex that gives the cohomology $$H_{log-crys}^i((X_0,M_0)/W^\times)\otimes K_0.$$

\subsection{Relative Log Convergent Cohomology}

Following \cite{[Sh08]}, we study the relative log convergent cohomology sheaves there defined. Again, we work only with the trivial isocrystal $\mathcal O_{X_A/\mathcal Y,conv}$, on the log convergent site, and denote the sheaves of relative cohomology by $$Rf_{X_A/\mathcal Y, conv *}(\mathcal O_{X_A/\mathcal Y,conv}).$$ \newline \indent 
Recall that there is a canonical functor (as in \cite{[Sh08]}) from the category of isocrystals on the relative log convergent site to that on the log crystalline site
\[
\Phi:I_{conv} ((X_A/\mathcal Y)_{conv}^{log}) \longrightarrow I_{crys} ((X_A/\mathcal Y)_{crys}^{log})
\]
sending locally free isocrystals on $(X_A/\mathcal Y)_{conv}^{log}$ to locally free isocrystals on $(X_A/\mathcal Y)_{crys}^{log}$. In particular, $\Phi(\mathcal O_{X_A/\mathcal Y,conv})=\mathcal O_{X_A/\mathcal Y,crys}$.\newline \indent
Now let us go back to the situation in diagram (\ref{bigdiagram}). Let $]X_\bullet[_{\mathcal P_\bullet}^{log}$ be the log tube of the closed immersion $i_\bullet$, and $\widehat{\mathcal P_\bullet}$ the completion of $\mathcal P_\bullet$ along $X_\bullet$. Then, as in \cite{[CT12]}, we have a specialization map
\[
\textrm{sp}:\textrm{}]X_\bullet[_{\mathcal P_\bullet}^{log} \rightarrow \widehat{\mathcal P_\bullet}.
\]
Moreover, if we denote by $\Omega_{]X_\bullet[_{\mathcal P_\bullet}^{log}/\mathcal Y_K}^\bullet \langle \mathcal M_\bullet / \mathcal N \rangle$
the logarithmic De Rham complex of the simplicial rigid analytic space $]X_\bullet[_{\mathcal P_\bullet}^{log}$ over the generic fiber $\mathcal Y_{K_0}$ of $\mathcal Y$, then by \cite[Corollary 2.34]{[Sh08]}, we have
\[
Rf_{X_A/\mathcal Y,conv *}(\mathcal O_{X_A/\mathcal Y, conv}) \cong R(f\theta)_* \textrm{sp}_*\Omega_{]X_\bullet[_{\mathcal P_\bullet}^{log}/\mathcal Y_{K_0}}^\bullet \langle \mathcal M_\bullet / \mathcal N \rangle.
\]
Now by using the remarks in \cite[p. 31]{[Sh08]} and passing to the projective limit, we have a canonical morphism of complexes
\begin{equation}\label{convcrys}
\textrm{sp}_*\Omega_{]X_\bullet[_{\mathcal P_\bullet}^{log}/\mathcal Y_{K_0}}^\bullet \langle \mathcal M_\bullet / \mathcal N \rangle \longrightarrow \lim_{\stackrel{\longleftarrow}{n}} C_{X_\bullet/\mathcal Y_n},
\end{equation}
which by \cite[Theorem 2.36]{[Sh08]} gives the following:
\begin{theorem}
The canonical morphism (\ref{convcrys}) induces an isomorphism
\[
R^m f_{X_A/\mathcal Y,conv *}(\mathcal O_{X_A/\mathcal Y, conv}) \cong R^m f_{X_A/\mathcal Y,crys *}(\mathcal O_{X_A/\mathcal Y,crys})
\]
of isocoherent sheaves on $\mathcal Y$.
\end{theorem}
In particular, by theorem \ref{finitecrys} this allows to prove that $Rf_{X_A/\mathcal Y,conv *}(\mathcal O_{X_A/\mathcal Y,conv})$ is a perfect complex of isocoherent sheaves, and a base change theorem:

\begin{theorem}
With the same notation as in diagram (\ref{basechange}), there is a natural isomorphism
\[
L\varphi^* R f_{X_A/\mathcal Y,conv *}(\mathcal O_{X_A/\mathcal Y,conv}) \cong R f_{s X_0/W, conv *}(\mathcal O_{X_0/ W,conv}).
\]
\end{theorem}
The complex $R f_{s X_0/W, conv *}(\mathcal O_{X_0/W,conv})$ gives the cohomology $$H_{log-conv}^i ((X_0,M_0)/W^\times).$$

\subsection{Relative Log Analytic Cohomology}
Now we study the sheaves of relative log analytic cohomology. Note that $g$ in diagram (\ref{bigdiagram}) induces a morphism $g_K^{ex}:]X_\bullet[_{\mathcal P_\bullet}^{log}\rightarrow \mathcal Y_K$. Then, the log analytic cohomolgy sheaves of $(X_A,M_A)/(Y,N)$ with respect to $(\mathcal Y,\mathcal N)$ can be computed by
\[
R^m f_{X_A/\mathcal Y,an *}(\mathcal O_{X_A/\mathcal Y,an}) = R^m g_{K *}^{ex}\Omega_{]X_\bullet[_{\mathcal P_\bullet}^{log}/\mathcal Y_{K_0}}^\bullet \langle \mathcal M_\bullet / \mathcal N \rangle.
\]
Then, by applying \cite[Theorem 4.6]{[Sh08]}, we have the following comparison theorem
\begin{theorem}
Let sp be the specialization map $\mathcal Y_K \rightarrow \mathcal Y$. Then for each $m$, $R^m f_{X_A/\mathcal Y, an *}(\mathcal O_{X_A/\mathcal Y,an})$ is a coherent sheaf on $\mathcal Y_K$, and there is an isomorphism
\[
\textrm{sp}_*R^m f_{X_A/\mathcal Y, an *}(\mathcal O_{X_A/\mathcal Y,an}) \cong R^m f_{X_A/\mathcal Y, conv *} (\mathcal O_{X_A/\mathcal Y,conv}).
\]
\end{theorem}

\section{Reduction to the case of a family over a curve}\label{Reduction}
Now that we have relative cohomology sheaves defined for the family over $Y$, we want to restrict those sheaves to a smaller family. Namely a family over a curve, in order to be in the same situation as in \cite{[CT12]}.\newline \indent
Let us first construct the curve that we shall use. As stated at the beginning of the preceding section, $A_0$ is a smooth $W$-algebra. Let $\widetilde Y=\Spec A_0$ and $S=\Spec W$. Since $Y\rightarrow \widetilde Y$ is a closed immersion, the image $\hat s$ of $s$ inside $\widetilde Y$ is a closed point. 
Since the natural morphism $\widetilde Y \rightarrow S$ is smooth, there exists an affine open neighborhood $\widetilde U$ of $\hat s$ and an \'etale morphism $\sigma: \widetilde U \rightarrow \mathbb A_W^{d}$ such that $\widetilde W \rightarrow S$ factors in the following way:
\[\begindc{0}[700]
\obj(0,1)[YY]{$\widetilde U$}
\obj(0,0)[SS]{$S$}
\obj(1,1)[Af]{$\mathbb A_W^{d}$}
\mor{YY}{Af}{$\sigma$}[\atleft,\solidarrow]
\mor{YY}{SS}{$$}[\atright,\solidarrow]
\mor{Af}{SS}{$$}[\atright,\solidarrow]
\enddc\]

Let us recall this construction. There exists an open affine subset $\widetilde U=\Spec (A_0)_g$ of $\widetilde Y$ such that the restriction of $\widetilde Y \rightarrow S$ is standard smooth. Moreover, we may assume (using the fact that the reduction modulo $p$ is smooth over $k[t]$) that we can write
\[
(A_0)_g=W[x_1,...,x_r,t]/(f_1,...,f_c),
\]
where the polynomial 
\[
\det \left[ \begin{array}{lcclll}
           \frac{\partial f_1}{\partial x_1} && \cdots && \frac{\partial f_c}{\partial x_1} \\
           \cdots && \cdots && \cdots \\
           \frac{\partial f_1}{\partial x_c} && \cdots && \frac{\partial f_c}{\partial x_c}
     \end{array}
\right]
\]
is invertible in $(A_0)_g$. Then, the morphism $W[x_{c+1},...,x_r,t]\rightarrow (A_0)_g$ is \'etale, and with $d=r+1-c$ we get the desired factorization.\newline \indent
Using this description it is clear how to construct a smooth curve $C_W$ inside $\widetilde U$, transversal to $(t=0)$ and passing through the point $\hat s$: by pulling back a curve with these properties inside $\mathbb A_W^d$. In particular its reduction $C$ modulo $p$ is a smooth curve inside $Y$, transversal to $(t=0)$ and passing throught the point $s$.\newline \indent
Let $N_C$ be the log structure on $C$ defined to make the closed immersion $(C,N_C)\rightarrow (Y,N)$ exact,  and then we have a sequence of exact closed immersions
\[
s^\times \rightarrow (C,N_C) \rightarrow (Y,N).
\]
Let $(X_C,M_C)=(X_A,M_A)\times_{(Y,N)}(C,N_C)$. Then, we have the following diagram, where all the squares are cartesian:

\[\begindc{0}[700]
\obj(0,1)[X0]{$(X_0,M_0)$}
\obj(0,0)[kk]{$s^\times$}
\obj(1,1)[XC]{$(X_C,M_C)$}
\obj(1,0)[CC]{$(C,N_C)$}
\obj(2,0)[YY]{$(Y,N)$}
\obj(2,1)[XA]{$(X_A,M_A)$}
\mor{X0}{kk}{$$}[\atleft,\solidarrow]
\mor{X0}{XC}{$$}[\atright,\solidarrow]
\mor{XC}{CC}{$$}[\atright,\solidarrow]
\mor{XC}{XA}{$$}[\atleft,\solidarrow]
\mor{kk}{CC}{}[\atleft,\solidarrow]
\mor{CC}{YY}{}[\atleft,\solidarrow]
\mor{XA}{YY}{}[\atleft,\solidarrow]
\enddc\]

Note that the family $(X_C,M_C) \rightarrow (C,N_C)$ is in the situation studied in \cite{[CT12]}. We denote by $\mathcal C$ the $p$-adic completion of $C_W$ along the special fiber $C$. Then $1\mapsto t$ defines a log structure $\mathcal N_{\mathcal C}$ on $\mathcal C$ and we have the following diagram

\[\begindc{0}[700]
\obj(0,1)[XC]{$(X_C,M_C)$}
\obj(0,0)[XA]{$(X_A,M_A)$}
\obj(1,1)[CC]{$(C,N)$}
\obj(1,0)[YY]{$(Y,N)$}
\obj(2,0)[CY]{$(\mathcal Y,\mathcal N)$}
\obj(2,1)[Cr]{$(\mathcal C, \mathcal N_{\mathcal C})$}
\mor{XC}{XA}{$$}[\atleft,\solidarrow]
\mor{XC}{CC}{$f_C$}[\atright,\solidarrow]
\mor{CC}{YY}{$$}[\atright,\solidarrow]
\mor{CC}{Cr}{$$}[\atleft,\solidarrow]
\mor{Cr}{CY}{$\iota$}[\atleft,\solidarrow]
\mor{YY}{CY}{}[\atleft,\solidarrow]
\mor{XA}{YY}{$f$}[\atleft,\solidarrow]
\enddc\]
Then by \cite[Theorem 1.19 and Corollary 2.38]{[Sh08]}, we have an isomorphism
\begin{equation}\label{1stisom}
L\iota^* Rf_{X_A/\mathcal Y, crys *}(\mathcal O_{X_A/\mathcal Y, crys}) \stackrel{\sim}{\longrightarrow} Rf_{C,X_C/\mathcal C,crys *} (\mathcal O_{X_C/\mathcal C,crys})
\end{equation}

Now consider the diagram

\[\begindc{0}[700]
\obj(0,1)[X0]{$(X_0,M_0)$}
\obj(0,0)[XC]{$(X_C,M_C)$}
\obj(1,1)[ss]{$s^\times$}
\obj(1,0)[CC]{$(C,N_C)$}
\obj(2,0)[Cr]{$(\mathcal C,\mathcal N_{\mathcal C})$}
\obj(2,1)[fs]{$\Spf W^\times$}
\mor{X0}{XC}{$$}[\atleft,\solidarrow]
\mor{XC}{CC}{$$}[\atright,\solidarrow]
\mor{X0}{ss}{$$}[\atright,\solidarrow]
\mor{ss}{CC}{$$}[\atleft,\solidarrow]
\mor{CC}{Cr}{$$}[\atleft,\solidarrow]
\mor{ss}{fs}{}[\atleft,\solidarrow]
\mor{fs}{Cr}{$\psi$}[\atleft,\solidarrow]
\enddc\]
where $\iota \circ \psi=\varphi$. Then we have an isomorphism
\begin{equation}\label{2ndisom}
L\psi^* Rf_{C,X_C/\mathcal C, crys *}(\mathcal O_{X_C/\mathcal C, crys}) \stackrel{\sim}{\longrightarrow} Rf_{s,X_0/W,crys *} (\mathcal O_{X_0/W,crys}).
\end{equation}
By combining the isomorphisms (\ref{1stisom}) and (\ref{2ndisom}), and the fact that $L\psi^* L \iota^* \cong L(\psi^*\iota^*) \cong L((\iota \circ \psi)^*)=L\psi^*$, we get that $Rf_{s,X_0/W,crys *} (\mathcal O_{X_0/W,crys})$ can be obtained from the family over $Y$ or over $C$. In particular, by the main result in \cite{[CT12]}, we get the following Clemens-Schmid type exact sequence: 
\begin{equation}\label{CSS}
\cdots \rightarrow H_{rig}^m(X_0) \rightarrow H_{log-crys}^m((X_0,M_0)/W^\times)\otimes K_0 \stackrel{N}{\rightarrow}  
\end{equation}
\[
H_{log-crys}^m((X_0,M_0)/W^\times)\otimes K_0(-1)\rightarrow H_{X_0,rig}^{m+2}(X_C) \rightarrow H_{rig}^{m+2}(X_0)\rightarrow \cdots
\]
The terms of the form $H_{X_0,rig}^{m+2}(X_C)$ depend a priori on the choice of the curve $C$, but if we choose a different smooth curve $C'$, by Poincar\'e duality \cite[Theorem 2.4]{[Be97]}, we have isomorphisms
\[
H_{X_0,rig}^{m+2}(X_C)\cong H_{X_0,rig}^{m+2}(X_{C'}) \cong H_{c,rig}^{2\dim X-m-2}(X_0)^\vee (-\dim X)
\] 
\[ \cong H_{2\dim X_0-m}^{rig}(X_0)(-\dim X),
\]
and we get a Clemens-Schmid type exact sequence that depends only on $X$ and the special fiber $X_0$ for our starting situation.

\section{Monodromy Criteria}\label{MonCrit}
As an application of the $p$-adic version of the Clemens-Schmid exact sequence, we prove a $p$-adic version of the Monodromy Criteria \cite[p.112]{[Mo84]}. We start with a situation in which he have an exact sequence of Clemens-Schmid type, as for example the situation in \cite{[CT12]}. Namely, suppose $k$ is a finite field and $C$ a smooth curve over $k$. We consider a proper and flat morphism $$f:X\rightarrow C,$$ where $X$ is a smooth variety of dimension $n+1$ over $k$. Moreover, we assume that there exists a $k$-rational point $s\in C$ such that the fiber of $f$ at $s$, which we denote by $X_s$ is a NCD. This defines a log structure $M$ on $X$. We denote by $(X_s,M_s)$ the log scheme with the induced log structure. \newline \indent
Then, the main result of \cite{[CT12]} states that there is a long exact sequence:

\[
\cdots \rightarrow H_{rig}^m (X_s) \rightarrow H_{log-crys}^m((X_s,M_s)/W^\times ) \otimes K_0 \rightarrow 
\]
\[
H_{log-crys}^m((X_s,M_s)/W^\times)\otimes K_0 (-1)\rightarrow H_{X_s,rig}^{m+2} (X)\rightarrow H_{rig}^{m+2}(X_s) \rightarrow \cdots
\]

We can consider the maps as morphisms of filtered vector spaces, where we give the weight filtration to each of them. Moreover, we know by the results in \cite[p.24]{[CT12]} that the weight filtration on the log-crystalline cohomology terms coincides with the monodromy one.\newline \indent
Now let us make a description of the filtration on $H_{rig}^m (X_s)$: denote by $X_1,...,X_r$ the irreducible components of $X_s$ and assume they are proper and smooth. Define the {\it codimension $p$ stratum of $X_s$} as
\[
X^{[p]}:=\bigsqcup_{i_0<\cdots < i_p} X_{i_0}\cap \cdots \cap X_{i_p}.
\] 

For each $a=0,...,p+1$, denote by $\delta_a : X^{[p+1]}\rightarrow X^{[p]}$ the natural map that restricted to each component is the inclusion
\[
X_{i_0} \cap \cdots \cap X_{i_{p+1}} \hookrightarrow X_{i_0} \cap \cdots \cap X_{i_{a-1}} \cap X_{i_{a+1}} \cap \cdots \cap X_{i_{p+1}}
\]
and define
\begin{equation}
\rho_p:=(-1)^{p}\sum_{a=0}^p (-1)^a \delta_a^*,
\end{equation}
where $\delta_a^*$ is the morphism of De Rham-Witt complexes $W_n \Omega_{X^{[p]}}^\bullet \rightarrow W_n \Omega_{X^{[p+1]}}^\bullet$ induced by $\delta_a$, where we identify $W_n \Omega_{X^{[p]}}^\bullet$ with its direct image in the \'etale site of $X_s$.\newline \indent
This gives a double complex
\begin{equation}
0\longrightarrow W_n \Omega_{X^{[0]}}^\bullet \stackrel{\rho_0}{\longrightarrow} W_n \Omega_{X^{[1]}}^\bullet \stackrel{\rho_1}{\longrightarrow} W_n \Omega_{X^{[2]}}^\bullet \stackrel{\rho_2}{\longrightarrow} \cdots
\end{equation}
and by taking projective limit, we get the double complex
\begin{equation}\label{dbcmplxrig}
0\longrightarrow W \Omega_{X^{[0]}}^\bullet \stackrel{\rho_0}{\longrightarrow} W \Omega_{X^{[1]}}^\bullet \stackrel{\rho_1}{\longrightarrow} W \Omega_{X^{[2]}}^\bullet \stackrel{\rho_2}{\longrightarrow} \cdots
\end{equation}
This allows to define a spectral sequence with
\begin{equation}\label{spsq}
E_1^{p,q}= H_{rig}^q(X^{[p]})
\end{equation}
with $d_1^{p,q}$ induced by $\rho_p$. 

\begin{theorem}\label{degconv}
The spectral sequence (\ref{spsq}) degenerates at $E_2$ and converges to $H_{rig}^* (X_s)$.
\end{theorem}
\begin{proof}
Since $X^{[p]}$ is smooth and proper, then $H_{rig}^q(X^{[p]})$ is pure of weight of $q$. Since $E_2^{p,q}$ is a sub quotient of this, we have that
\[
d_2^{p,q}: E_2^{p,q} \rightarrow E_2^{p+2,q-1}
\]
has to be the zero morphism, which proves the degeneracy. To prove that it converges to $H_{rig}^* (X_s)$ it is enough to notice that the simple complex associated to (\ref{dbcmplxrig}) gives this cohomology. This is given by \cite[Proposition 1.8 and Theorem 3.6]{[Ch99]}.
\end{proof}

The weight filtration on rigid cohomology (given by the Frobenius operator) is induced by the spectral sequence (\ref{spsq}). Now we list some properties of this filtration, denoted by $W_\bullet $, and its respective graded modules $Gr_\bullet$ on $H_{log-crys}^m :=H_{log-crys}^m((X_s,M_s)/W^\times ) \otimes K  $ and $H_{rig}^m:=H_{rig}^m(X_s)$, which are just a consequence of the previous remarks and theorem \ref{degconv}.

\begin{proposition}\label{Proposition 1}
\begin{enumerate}[(i)]
\item $N^k$ induces an isomorphism of vector spaces $$Gr_{m+k} H_{log-crys}^m \stackrel{\sim}{\longrightarrow} Gr_{m-k} H_{log-crys}^m$$ for all $k\geq 0$.
\item For $k\leq m$, we have a decomposition
\[
Gr_k(H_{log-crys}^m) = \bigoplus_{a=0}^{[k/2]} Gr_{k-2a} (\mathcal K_m),
\]
where $\mathcal K_m=\ker N \subset H_{log-crys}^m$ is the kernel of $N$ acting on $H_{log-crys}^m$, and the filtration on $\mathcal K_m$ is induced by the one on $H_{log-crys}^m$.
\item  $Gr_0(H_{rig}^m)=H^m(|\Gamma|)$, where $\Gamma$ is the dual graph associated to $X_s$.
\item $Gr_k(H_{rig}^m)=E_2^{m-k,k}$.
\end{enumerate}
\end{proposition}

We also need the following, which is an immediate corollary of the Clemens-Schmid exact sequence. 

\begin{proposition}\label{Proposition 2}
For all $k<m$, $W_k(H_{rig}^m) \cong W_k(\mathcal K_m)$.
\end{proposition}
\begin{proof}
It is enough to note that since $H_{X_s,rig}^m (X)$ has weights $>m-1$, when restricting the Clemens-Schmid sequence to the $W_k$-parts we get an exact sequence:
\[
0\rightarrow W_k(H_{rig}^m) \rightarrow W_k(\mathcal K_m) \rightarrow 0
\]
for $k<m$.
\end{proof}

With the two previous propositions in hand, we can prove the following monodromy criteria:

\begin{theorem}\label{MonodromyCriteriapadic} 
Denote $H_{log-crys}^i:=H_{log-crys}^i((X_s,M_s)/W^\times)\otimes K$.  Let $h^k(|\Gamma|)$$=\dim H^k(|\Gamma|)$, $b_k(X_s)=\dim H_{log-crys}^k$, $h^k(X^{[j]})=\dim H_{rig}^k(X^{[j]})$ and $\Phi=\dim Gr_1 H_{rig}^1$. Then, we have the following:
\begin{enumerate}[(i)]
\item $N=0 $ on $H_{log-crys}^1$ if and only if  $h^1(|\Gamma|)=0$ if and only if $b_1(X_s)=\Phi$.
\item $N^2=0$ on $H_{log-crys}^2$ if and only if $h^2(|\Gamma|)=0$.
\item $N=0$ on $H_{log-crys}^2$ if and only if $h^2(|\Gamma|)=0$ and $\Phi=h^1(X^{[0]})-h^1(X^{[1]})$
\end{enumerate}
\end{theorem}

\begin{proof}
\begin{enumerate}[(i)]
\item By the final remark in \cite{[Ch99]}, we have an exact sequence
\[
0\rightarrow H_{rig}^1 \rightarrow H_{log-crys}^1 \stackrel{N}{\rightarrow} H_{log-crys}^1.
\]
In particular, $\ker N \cong H_{rig}^1$. Then, by part (i) of proposition \ref{Proposition 1}, we have $Gr_2 H_{log-crys}^1 \cong Gr_0 H_{log-crys}^1$, and by part (ii), we have $Gr_0 H_{log-crys}^1 \cong Gr_0 (\mathcal K_1)=Gr_0 (H_{rig}^1)$, and by part (iii), we conclude that $Gr_2 H_{log-crys}^1 \cong H^1 (|\Gamma|)$. \newline \indent
Similarly, by part (ii) of proposition \ref{Proposition 1}, we have $$Gr_1 H_{log-crys}^1 \cong Gr_1 (\mathcal K_1) = Gr_1 H_{rig}^1.$$
First suppose $N=0$. Then, $Gr_2 H_{log-crys}^1 = 0=Gr_0 H_{log-crys}^1$, since the first isomorphism is induced by $N$. Then it follows that $h^1(|\Gamma|)=0$ and $b_1(X_s)=\Phi$.\newline \indent
Now suppose that $h^1(|\Gamma|)=0$. Then, $Gr_0 H_{log-crys}^1 \cong Gr_0 (H_{rig}^1) = 0$. This implies that $Gr_1 H_{log-crys}^1 = H_{log-crys}^1$, but $Gr_1 H_{log-crys}^1 = Gr_1 \mathcal K_1$, hence $Gr_1 \mathcal K_1 = H_{log-crys}^1$. By part (ii) of proposition \ref{Proposition 1}, we also have $Gr_0 \mathcal K_1 \cong Gr_0 H_{log-crys}^1 =0 $ and $$Gr_2 \mathcal K_1 \oplus Gr_0 \mathcal K_1 = Gr_2 \mathcal K_1 \cong Gr_2 H_{log-crys}^1 =0, $$ hence $\mathcal K_1=Gr_1 \mathcal K_1 = H_{log-crys}^1$, which proves that $N=0$. \newline \indent
Finally, note that if $b_1(X_s)=\Phi$, then $Gr_1 H_{log-crys}^1 = H_{log-crys}^1$, and this implies that $h^1(|\Gamma|)=0$.

\item
For the proof of this and next part, we note that the Clemens-Schmid sequence for even indices can be seen as two exact sequences (since $N=0$ on $H_{log-crys}^0$):
$$
0 \rightarrow H_{rig}^0 \rightarrow H_{log-crys}^0 \rightarrow 0
$$
$$
0 \rightarrow H_{log-crys}^0 \rightarrow H_{X_s,rig}^2(X) \rightarrow H_{rig}^2 \rightarrow H_{log-crys}^2 \stackrel{N}{\rightarrow} H_{log-crys}^2 \rightarrow \cdots
$$

By part (ii) of proposition \ref{Proposition 1}, we have that $Gr_0 H_{log-crys}^2 \cong Gr_0 \mathcal K_2$, and by proposition \ref{Proposition 2}, this is isomorphic to $Gr_0 H_{rig}^2\cong H^2 (|\Gamma|)$. \newline \indent
Suppose that $N^2=0$ on $H_{log-crys}^2$. Then, by part (i) of proposition \ref{Proposition 1}, we have $Gr_4 H_{log-crys}^2 \cong Gr_0 H_{log-crys}^2 = 0$, and this gives that $h^2(|\Gamma|)=0$. \newline \indent
Conversely, suppose that $h^2(|\Gamma|)=0$. Then, $\dim Gr_0 H_{rig}^2=0$. Note that $N^2$ takes $W_0$ to $W_{-4}=0$, $W_1$ to $W_{-3}=0$, $W_2$ to $W_{-2}=0$, $W_3$ to $W_{-1}=0$ and $W_4$ to $W_0=Gr_0H_{rig}^2=0$. Thus, $N^2=0$.

\item
By part (ii) of proposition \ref{Proposition 1}, we have that $Gr_1 H_{log-crys}^2 \cong Gr_1 \mathcal K_2$, and by proposition \ref{Proposition 2}, this is isomorphic to $Gr_1 H_{rig}^2=E_2^{1,1}=\ker d_1^{1,1}/\mbox{ Im }d_1^{0,1}$.\newline \indent
Note that $d_1^{1,1}:H_{rig}^1 (X^{[1]}) \rightarrow H_{rig}^2 (X^{[1]})$ is the zero map (since $H_{rig}^2 (X^{[1]})$ is trivial). Then, we conclude that
\[
\dim Gr_1 H_{log-crys}^2= h^1(X^{[1]}) - \dim \mbox{Im } (H_{rig}^1(X^{[0]})\rightarrow H_{rig}^1(X^{[1]}))
\]
\[
=h^1(X^{[1]}) -(h^1(X^{[0]})- \dim \ker (H_{rig}^1(X^{[0]})\rightarrow H_{rig}^1(X^{[1]})) )
\]
\[
=\Phi-h^1(X^{[0]})+h^1(X^{[1]}).
\]
Now suppose that $N=0$. Then, $N^2=0$ and by the preceding part, we have $h^2(|\Gamma|)=0$. Moreover, $N$ induces an isomorphism
\[
Gr_3(H_{log-crys}^2) \stackrel{\sim}{\longrightarrow} Gr_1(H_{log-crys}^2),
\]
hence $Gr_1 H_{log-crys}^2 = 0$ and $\Phi-h^1(X^{[0]})+h^1(X^{[1]})=0$.\newline \indent
Conversely, suppose that $h^2(|\Gamma|)=0$ and $\Phi-h^1(X^{[0]})+h^1(X^{[1]})=\dim Gr_1 H_{log-crys}^2=0$ and let us prove that $\mathcal K_2=H_{log-crys}^2$ (hence $N=0$). First note that $Gr_0 H_{log-crys}^2 =0$, since $h^2(|\Gamma|)=0$, i.e., $W_0=0$. But since $Gr_1 H_{log-crys}^2=0$, then $W_1=0$. By part (ii) of proposition \ref{Proposition 1}, we have $Gr_3 H_{log-crys}^2=0$, hence $W_3=W_2$. By the same argument, $Gr_0 H_{log-crys}^2 \cong Gr_4 H_{log-crys}^2$, hence $W_4=W_3=W_2=H_{log-crys}^2$. This gives $Gr_2 H_{log-crys}^2=H_{log-crys}^2$. By part (ii) of proposition \ref{Proposition 1}, we get
\[
Gr_2 H_{log-crys}^2 = Gr_2 \mathcal K_2 \oplus Gr_0 \mathcal K_2 = Gr_2 \mathcal K_2 = \mathcal K_2,
\]
which concludes the proof.
\end{enumerate}
\end{proof}

\section{Application to $K3$ surfaces}\label{AppK3}
In this section we assume $p>3$. Let $K$ be a finite extension of $\mathbb Q_p$ and denote by $O_K$ its ring of integers, $\pi$ a uniformizer of $O_K$ and $k$ its residue field. We consider a smooth, projective $K3$ surface $X_K$ over $K$, i.e., a smooth, projective surface with trivial canonical sheaf and irregularity 0. Moreover, we assume that $X_K$ has a semi-stable model $X \rightarrow \Spec O_K$, i.e., $X$ is a proper scheme over $O_K$, \'etale locally \'etale over a scheme of the form $$\Spec (O_K[x_1,...,x_n]/(x_1\cdots x_r-\pi)).$$ Let $X_s:=X\otimes_{O_K}k$ be the special fiber of $X$ and assume it is a combinatorial $K3$ surface. In particular, we may assume that we are in one of the following cases:
\begin{enumerate}[I)]
\item $X_s$ is a smooth $K3$ surface over $k$
\item $X_s=X_0\cup X_1 \cup \cdots \cup X_{j+1}$ is a chain of smooth surfaces, with $X_0, X_{j+1}$ rational and the others are elliptic ruled and double curves on each of them are rulings.
\item $X_s=X_0 \cup X_1 \cup \cdots \cup X_{j+1}$ is a chain of smooth surfaces and every $X_i$ is rational, and the double curves on $X_i$ are rational and form a cycle on $X_i$. The dual graph of $X$ is a triangulation of the sphere $S^2$.

\end{enumerate}

We shall refer to each of these as surface of type I, II and III, respectively. 
\begin{remark}
The definition of a combinatorial $K3$ surface \cite[Definition 3.2]{[Na00]} requires for the cases II) and III) that the geometric special fiber $X_{\overline s}$ has a decomposition of those types and not necessarily $X_s$, but this implies that there exists a finite extension $k'$ of $k$ such that the base change $X_{k'}=X_s \otimes_k k'$ has such decomposition. Since $k'$ is again a finite field, we may assume that it is $X_s$ the one that admits such decomposition.
\end{remark}

\begin{remark}
The condition of having a combinatorial special fiber is achieved for example if the model satisfies $\omega_{X/O_K} = 0$. If we take any general semistable model of $X_K$, one can apply Kawamata's minimal model program (only in the case $p>3$, see \cite{[Kw93]},\cite{[Kw98]}) and we get the same situation, but with a model which is an algebraic space and not necessarily a scheme (even if the fibers are schemes themselves). 
\end{remark}

\begin{remark}\label{RationalRmk}
To study surfaces of type II we use the fact that for a smooth, proper, rational surface $Y$ over a field (such as $X_0$ and $X_{j+1}$), we have $H_{rig}^1(Y) $ is zero. Indeed, first note that since $Y$ is smooth and proper over a field, then $Y$ is necessarily projective (see \cite[Remark 3.5, Ch.9]{[Liu]}). Then, we use Castelnuovo-Zariski's criterion in characteristic $p$ as stated in \cite[Theorem 4.6]{[Li13]} to get that the first $\ell$-adic \'etale cohomology group is trivial. Since $Y$ is smooth and proper, we conclude that the dimension of the first rigid cohomology group is also 0, since rigid cohomology is a Weil cohomology (see \cite{[Ch98]}). 
\end{remark}

Recall from \cite[Section 2]{[Na00]} that the special fiber $X_s$ can be endowed with a log structure $M_s$ in such a way that we have a log smooth morphism $(X_s,M_s)\rightarrow (\Spec k, \mathbb N^m),$ where $m$ is the number of connected components of the singular locus of $X_s$ and the log structure is defined by $e_i \mapsto 0$, where $e_i$ denotes the $i$th canonical generator of $\mathbb N^m$. \newline \indent 
Let us make an explicit description of $M_s$. In general, suppose that $Y$ is a normal crossing variety and denote by $Y_{\textrm{sing}}$ the singular locus. Denote by $Y_1,...,Y_m$ the connected components. For each $i=1,...,m$, we can endow $\Spec (k[x_0,...,x_n]/(x_0\cdots x_r))$ with a log structure given by as follows:
\[
\mathbb N^{m+r}=\mathbb N^{i-1}\oplus \mathbb N^{r+1} \oplus \mathbb N^{m-i} \rightarrow k[x_0,...,x_n]/(x_0\cdots x_r)
\]
\[
e_i \mapsto \left\{ \begin{array}{lcll} 
0 && \textrm{ if } e_i \in \mathbb N^{i-1} \\
x_{i-1} && \textrm{ if } e_i \in \mathbb N^{r+1} \\
0 && \textrm{ if } e_i \in \mathbb N^{m-i}
\end{array}
\right.
\]

Then,

\begin{enumerate}
\item If $x$ is a smooth point of $Y$, \'etale locally on a neighbourhood of $x$, the log structure is the pull-back of the log structure of the log-point $(\Spec k, \mathbb N^m)$
\item If $x\in Y_i$, \'etale locally on a neighbourhood of $x$, the log structure is the pull-back of the log structure defined above.
\end{enumerate}
Since $X_s$ is in particular a normal crossing variety over $k$, we can endow $X_s$ with this log structure and we denote it by $M_s$. Note that this is not the usual log structure defined for example in \cite{[Ka89]}, which we denote here by $M_s'$. As it is stated in \cite[p.358]{[Na00]}, the relationship between them is
\[
(X_s,M_s') = (X_s,M_s) \times_{(\Spec k, \mathbb N^m)} (\Spec k, \mathbb N),
\]
where the morphism of log schemes $(\Spec k, \mathbb N) \rightarrow (\Spec k, \mathbb N^m)$ is defined by $s:\mathbb N^m \rightarrow \mathbb N$ the sum of the components. Moreover, the sheaves of relative log differentials $\omega_{(X_s,M_s)/(\Spec k, \mathbb N^m)}^{\bullet}$ and $\omega_{(X_s,M_s')/(\Spec k, \mathbb N)}^{\bullet}$ coincide, and there is also a canonical isomorphism 
\[
H_{log-crys}^i ((X_s,M_s)/(W,\mathbb N^m)) \cong H_{log-crys}^i ((X_s,M_s')/(W,\mathbb N)),
\] 
as stated and proved in \cite{[Na00]}.\newline \indent
Assume for the moment that $X_s$ is either of type I), type III) or type II) such that the double curve is ordinary. Then by \cite[Corollary 5.4, Proposition 5.9]{[Na00]}, there exists a semistable family $X^{\log}$ over $\Spec k[[t]]^{\log} $ such that its special fiber is precisely $(X_s,M_s)$. Then we have the following diagram, with cartesian squares:
\begin{equation}\label{bigdiagram2}
\begindc{\commdiag}[700]
\obj(0,1)[Xs']{$(X_s,M_s')$}
\obj(0,0)[kk]{$(\Spec k,\mathbb N)$}
\obj(1,1)[Xs]{$(X_s,M_s)$}
\obj(1,0)[km]{$(\Spec k, \mathbb N^m)$}
\obj(2,1)[Xl]{$X^{\log}$}
\obj(2,0)[ss]{$\Spec k[[t]]^{\log}$}

\mor{Xs'}{kk}{$$}[\atleft,\solidarrow]
\mor{Xs'}{Xs}{$$}[\atright,\solidarrow]
\mor{kk}{km}{$$}[\atright,\solidarrow]
\mor{Xs}{km}{$$}[\atleft,\solidarrow]
\mor{Xs}{Xl}{$$}[\atleft,\solidarrow]
\mor{Xl}{ss}{$$}[\atleft,\solidarrow]
\mor{km}{ss}{$$}[\atleft,\solidarrow]

\enddc
\end{equation}
Let $\widetilde X$ be the underlying scheme of $X^{\log}$. Then we can apply the same technique as in section \ref{Reduction} and get a smooth curve $C$ over $k$, and a regular scheme $X_C$ with a proper, flat morphism $X_C \rightarrow C$ such that there exists a $k$-rational point $s\in C$ such that the fiber of $f$ at $s$ is precisely $X_s$. Then, we have the following: 
\begin{theorem}\label{mainthm}
\begin{enumerate}[(a)]
\item $X_s$ is of type I if and only if $N=0$ on $H_{log-crys}^2$.
\item $X_s$ is of type II if and only if $N\neq 0$ and $N^2=0$ on $H_{log-crys}^2$.
\item $X_s$ is of type III if and only if $N^2 \neq 0$ on $H_{log-crys}^2$.
\end{enumerate}
\end{theorem}

\begin{proof}
We shall prove that if $N=0$ on $H_{log-crys}^2$, then $X_s$ is neccesarily of type I; if $N\neq 0$ and $N^2=0$, then $X_s$ is necessarily of type II; and if $N^2 \neq 0$, then $X_s$ is necessarily of type III. This shall prove the equivalence, since we know that we can be only in one of these three cases.\newline \indent
First assume that $X_s$ is of type I. Then, $X^{[0]}=X_s$, $X^{[1]}=\varnothing$ and the dual graph $\Gamma$ is only one point. In this case, the spectral sequence has the form
\[
E_{\infty}^{p,q}=E_1^{p,q}=H_{rig}^q (X^{[p]})=
\left\{
\begin{array}{lcll}
0 && \textrm{if } p \geq 1 \\
H_{rig}^q(X_s) && \textrm{if } p=0
\end{array}
\right.
\]
and this gives immediately that $\Phi=\dim Gr_1 H_{rig}^1 = \dim E_2^{0,1}=0$. Since $H_{rig}^1 (X_s)=H_{rig}^1(X^{[1]})=0$, and $h^2(|\Gamma|)=0$, we conclude that $N=0$, by theorem \ref{MonodromyCriteriapadic} (iii). \newline \indent

Now assume that $X_s$ is of type II (use the same notation as in the beginning of the section). In this case, it is clear that the dual graph is homeomorphic to $[0,1]$. In particular, $h^2(|\Gamma|)=0$ and $N^2=0$ by theorem \ref{MonodromyCriteriapadic} (ii). 
By definition of the type II, $X^{[1]}$ is the disjoint union of $j+1$ elliptic curves, hence $h^1(X^{[1]})=2j+2$. Since $X_0$ and $X_{j+1}$ are rational surfaces, by lemma \ref{RationalRmk}, we have
\[
h^1(X^{[0]})=\sum_{i=1}^j h^1(X_i),
\]
but the $X_i$'s are ruled, with the double curves rulings. Then, $h^1(X^{[0]})=2j$ and we get $h^1(X^{[0]})-h^1(X^{[1]})=-2$, but $\Phi$ cannot be negative, hence $$\Phi \neq h^1(X^{[0]})-h^1(X^{[1]}) $$ and $N\neq 0$.
Finally, assume that $X_s$ is of type III. In this case, $h^2(|\Gamma|)=h^2(S^2)=1 \neq 0$, hence $N^2 \neq 0$.\newline \indent
The only remaining case is when $X_s$ is of type II such that the double curve is not ordinary, i.e., supersingular. In this case, by \cite[Corollary 6.9]{[Na00]},  the geometric special fiber $X_{\overline s}$ is the special fiber of a projective semistable family $\widetilde X$ over $\Spec \overline k [[t]]$. Now we use the same approximation argument from section \ref{Popescu}, and we get the following cartesian diagram: 

\[
\begindc{0}[700]
\obj(0,1)[XX]{$\widetilde X$}
\obj(0,0)[VV]{$\Spec \overline k[[t]]$}
\obj(1,1)[XA]{$X_A$}
\obj(1,0)[YY]{$\text{Spec } A$}
\mor{XX}{XA}{}[\atleft,\solidarrow]
\mor{XX}{VV}{}[\atright,\solidarrow]
\mor{VV}{YY}{}[\atright,\solidarrow]
\mor{XA}{YY}{}[\atleft,\solidarrow]
\enddc
\]
where $A$ is a smooth $\overline k[t]$-algebra. By properness of $X_A$ over $A$ and proposition \ref{Iovitaremark}, there exists a finite extension $k'$ of $k$ and a $k'[t]$-algebra $A'$ over which we can define $X_{A'}$ to have a cartesian diagram

\[
\begindc{0}[700]
\obj(0,1)[XX]{$X_A$}
\obj(0,0)[VV]{$\Spec A$}
\obj(1,1)[XA]{$X_{A'}$}
\obj(1,0)[YY]{$\text{Spec } A'$}
\mor{XX}{XA}{}[\atleft,\solidarrow]
\mor{XX}{VV}{}[\atright,\solidarrow]
\mor{VV}{YY}{}[\atright,\solidarrow]
\mor{XA}{YY}{$f'$}[\atleft,\solidarrow]
\enddc
\]
The composition $A'\rightarrow A \rightarrow \overline k[[t]] \rightarrow \overline k$ defines a closed point $x$ in $\Spec A'$. Then, the fiber of $f'$ at $x$, denoted by $X_x$, satisfies
\[
X_x \otimes_{k'} \overline k \cong X_s \otimes_k \overline k = X_{\overline s}
\]  
Then, there exists a finite extension $k''$ of $k'$ such that $$X_x \otimes_{k'} k'' \cong X_s \otimes_{k} k''=:X_s''.$$ Since $(X_s \otimes_k k'') \otimes_{k''} \overline k = X_s'' \otimes_{ k''} \overline k$, we get that $X_s$ is of the same type (I, II or III) as $X_s''$. Moreover, if we denote by $K''/K$ the extension corresponding to $k''/k$, then the degree of nilpotency on $H_{log-crys}^2$ and $H_{log-crys}^2 \otimes_K K''$ is preserved, since any extension of fields is faithfully flat. This completes the proof for all the cases.
\end{proof}

A direct consequence of this theorem is the following, which is the desired good reduction criterion:
\begin{corollary}\label{MainCorollary}
Let $p>3$ and $K$ a finite extension of $\mathbb Q_p$. Let $X_K$ be a smooth, projective $K3$ surface over $K$, that admits a semistable model over $O_K$. Then, $X_K$ has good reduction if and only if the monodromy operator $N$ on $H_{DR}^2(X_K)$ is zero.
\end{corollary}

 \noindent
Genaro Hern\'andez Mada. Departamento de Matematicas. Universidad de Sonora. Blvd Luis Encinas y Rosales. Hermosillo, Son. Mexico 83240. E-mail: genarohm@mat.uson.mx

\end{document}